\documentclass[a4paper]{amsart}
\usepackage{amssymb}
\input xy
\xyoption{all}

\newcommand{\R}{\mathbb R}

\newcommand{\F}{\mathcal F}

\newcommand{\co}{\mathfrak{c}}

\newtheorem{theorem}{Theorem}
\newtheorem{df}{Definition}
\newtheorem{lemma}{Lemma}

\newtheorem{problem}{Problem}
\newtheorem{proposition}{Proposition}

\begin{document}

\title{Fuzzy integrals on compacta based on pseudo-grouping  and overlap functions}


\author{Taras Radul}

\maketitle

Institute of Mathematics, Kazimierz Wielki University in Bydgoszcz, Poland;
\newline
Department of Mechanics and Mathematics, Ivan Franko National University of Lviv,
Universytettska st., 1. 79000 Lviv, Ukraine.
\newline
e-mail: tarasradul@yahoo.co.uk



\begin{abstract} We consider a wide family of fuzzy integrals on arbitrary compactum  which generalize well know Sugeno integral. Such generalization is obtained using some non-discrete analogs of pseudo-grouping  functions.

 \end{abstract}

\maketitle

\section{Introduction}

Capacities (non-additive measures, fuzzy measures) were introduced by Choquet in \cite{Ch} as a natural generalization of additive measures. They found numerous applications (see for example \cite{EK},\cite{Gil},\cite{Sch}). Capacities on compacta were considered in \cite{Lin} where the important role plays the upper-semicontinuity property which  connects the capacity theory with the topological structure.

In fact, the most of applications of non-additive measures to game theory, decision making theory, economics etc deal not with measures as set functions  but with integrals which allow to obtain expected utility or expected pay-off.  Several types of integrals with respect to non-additive measures were developed for different purposes (see for example \cite{Grab}, \cite{KM}, \cite{LMOS}, \cite{Den}). Such integrals are called fuzzy integrals. The most known are the Choquet integral based on the addition and the multiplication operations \cite{Ch} and the Sugeno integral  based on the maximum and the minimum operations \cite{Su}.

Important types of aggregation functions have been researched in recent years, including generalizations of Sugeno and Choquet integrals. Such generalizations are based on changing above mentioned basic operations for more general classes of aggregations functions like fusion functions,   overlap functions, t-norms, t-conorms, grouping functions etc (see for example \cite{BK}, \cite{LMOS}, \cite{KM},   \cite{Sua}, \cite{We2}).  Many of these generalizations consider only discrete case: integrals of functions with finite domain. Thus, aggregation function under such approach are n-ary operations. We consider some  generalizations of the Sugeno integral for any compactum in this paper. We use specific classes of  aggregation functions called t-overlap and
$\co$-general pseudo-grouping functions.

\section {Preliminaries. Capacities   and t-normed integrals}

In what follows, all spaces are assumed to be compacta (compact Hausdorff space) except for $\R$ and the spaces of continuous functions on a compactum. All maps are assumed to be continuous. By $\F(X)$ we denote the family of all closed subsets of a compactum $X$.

We shall denote the
Banach space of continuous functions on a compactum  $X$ endowed with the sup-norm by $C(X)$. For any $c\in\R$ we shall denote the
constant function on $X$ taking the value $c$ by $c_X$. We also consider the natural lattice operations $\vee$ and $\wedge$ (on $C(X)$ and  its sublattice $C(X,[0,1])$).

We need the definition of capacity on a compactum $X$. We follow a terminology of \cite{NZ}.

\begin{df} A function $\nu:\F(X)\to [0,1]$  is called an {\it upper-semicontinuous capacity} on $X$ if the three following properties hold for each closed subsets $F$ and $G$ of $X$:

1. $\nu(X)=1$, $\nu(\emptyset)=0$,

2. if $F\subset G$, then $\nu(F)\le \nu(G)$,

3. if $\nu(F)<a$ for $a\in[0,1]$, then there exists an open set $O\supset F$ such that $\nu(B)<a$ for each compactum $B\subset O$.
\end{df}

By $MX$ we denote the set  of all upper-semicontinuous  capacities on a compactum $X$. Since all capacities we consider here are upper-semicontinuous, in the following we call elements of the set $MX$ simply capacities.

 Remind that a triangular norm $\ast$ is a binary operation on the closed unit interval $[0,1]$ which is associative, commutative, monotone and $s\ast 1=s$ for each  $s\in [0,1]$ \cite{PRP}. 
 We consider only continuous t-norms in this paper.

Integrals generated  by  t-norms are called t-normed integrals and were studied in \cite{We1}, \cite{We2} and \cite{Sua}. Denote $\varphi_t=\varphi^{-1}([t,1])$ for each $\varphi\in C(X,[0,1])$ and $t\in[0,1]$. So, for a continuous t-norm $\ast$, a capacity $\mu$ and a  function $f\in  C(X,[0,1])$ the corresponding t-normed integral is defined by the formula $$\int_X^{\vee\ast} fd\mu=\max\{\mu(f_t)\ast t\mid t\in[0,1]\}.$$ Let us remark that putting $\ast=\wedge$, we obtain the definition of the Sugeno integral.

\section {Overlap and general pseudo-grouping functions} In this section, we consider some modification of   general pseudo-grouping functions  and some generalization of the  binary overlap function.

\begin{df}\label{Ov} A continuous map $O:[0,1]^2\to [0,1]$ is an t-overlap function if it satisfies the following conditions:

1. $O$ is symmetric;

2. $O(l,s)=0$ if $ls=0$;

3. $O(l,s)=1$ iff $ls=1$;

4. $O$ is nondecreasing.

\end{df}

The minimum operation is one of the examples of  t-overlap function. Let us remark that the previous definition is a slight generalization of the overlap function from \cite{Overlap}.  The only difference is that the overlap function $o:[0,1]^2\to [0,1]$ satisfies $o(l,s)=0$ if and only if $ls=0$. Let us remark that not each continuous  t-norm satisfies that condition. But each continuous  t-norm is an t-overlap function.

The definition of the general pseudo-grouping function was given in \cite{OverGPG} as $n$-ary operation on the unit interval  and was used to define some discrete fuzzy integrals.

\begin{df}\cite{OverGPG}\label{GPGn} A continuous map $G:[0,1]^n\to [0,1]$ is a general pseudo-grouping function if it satisfies the following conditions:

1. $G(\{x_1,\dots x_n\})=0$, if $x_i=0$ for each $i$;

2. $G(\{x_1,\dots x_n\})=1$, if there exists $i$ such that $x_i=1$;

3. $G$ is nondecreasing in each variable.

\end{df}

The maximum operation is one of the examples of general pseudo-grouping function. But we will consider generalization of Sugeno integral on compacta and $n$-arity of  general pseudo-grouping function is not enough for this purpose. Let us remark that the maximum operation was considered in more general context in the definition of t-normed integral given in previous Section.

Let $X$ be a compactum. A function $f:X\to\R$ is called upper semi-continuous if $f^{-1}((-\infty,b))$ is open for each $b\in\R$.  It is well known that upper semi-continuous functions  reach its maximum on compacta.

\begin{lemma}\label{semi} Let $X$ be a compactum, $\mu\in MX$, $f\in C(X,[0,1])$  and $O:[0,1]^2\to [0,1]$ is a continuous nondecreasing function. Then the function  $m:[0,1]\to[0,1]$ defined by the formula $m(t)=O(\mu(f_t),t)$ is upper semi-continuous.
\end{lemma}

\begin{proof} Consider any $s>0$ and $t\in[0,1]$ such that $O(\mu(f_t),t)<s$. Since the function $O$ is continuous and  nondecreasing, there exist $\varepsilon>0$ such that for each $p$, $l\in [0, \mu(f_t)+\varepsilon)\times[0,t+\varepsilon)$ we have $O(p,l)<s$. Since $\mu$ is upper semi-continuous, there exists a neighborhood $V$ of $f_t$ such that for each closed set $A\subset V$ we have $\mu(A)<\mu(f_t)+\varepsilon$. Since $f$ is continuous and $X$ is compact, there exists $\delta>0$ such that $f_z\subset V$ for each $z\in (t-\delta,t+\delta)$. Put $\beta=\min\{\varepsilon,\delta\}$. Then we have $m(q)<s$ for each $q\in(t-\beta,t+\beta)$. 
\end{proof}

The previous lemma indicates how we should generalize the notion of general pseudo-grouping function.

For a compactum $X$ by $\exp X$ we denote the set of non-void
compact subsets of $X$ provided with the Vietoris topology. A subbase
of this topology consists of the sets of the form
$\langle U\rangle=\{A\in \exp\ X\mid A\subset U\}$ and $\langle X,V\rangle=\{A\in \exp\ X\mid
A\cap V\ne\emptyset$  where $U$,
$V$ are open in $X$. The space $\exp\ X$ is called the
hyperspace of $X$.

For a compactum $X$ by $UC(X,[0,1])$ we denote the set of all upper semi-continuous functions from $X$ to $[0,1]$. It is well known that $f:X\to[0,1]$ is u.s.c.
on X if and only if its hypograph, the set $$\Gamma(f)= \{(x, a)\mid x\in X \text{ and }
a\le f(x)\}$$ is a closed subset of $X\times [0,1]$. Thus, if we identify an u.s.c.
function with its hypograph, then the set of all u.s.c. functions
$UC(X)$ on $X$ can be viewed as a subspace of $\exp (X\times [0,1])$.  

\begin{df}\label{CGPG}  A continuous map $G:UC([0,1],[0,1])\to [0,1]$ is a $\co$-general pseudo-grouping function if it satisfies the following conditions:

1. $G(0_{[0,1]})=0$;

2. $G(\varphi)=1$  for each $\varphi\in UC([0,1],[0,1])$ such that  there exists $t\in [0,1]$ such that $\varphi(t)=1$;

3. $G$ is nondecreasing, i.e $G(\psi)\le G(\varphi)$ for each $\psi$, $\varphi\in UC([0,1],[0,1])$ such that $\psi\le \varphi$.

\end{df}

Let us check that the notion of $\co$-general pseudo-grouping function generalizes the maximum operation.

\begin{proposition} The map $\max:UC([0,1],[0,1])\to [0,1]$ is a $\co$-general pseudo-grouping function.
\end{proposition}

\begin{proof} Obviously, $\max$ satisfies all three properties from the definition of $\co$-general pseudo-grouping function. We have to prove continuity.

Consider any function $f\in UC([0,1],[0,1])$ and put $a=\max f\in [0,1]$. Take any $\varepsilon>0$. We have $\Gamma(f)\in\langle [0,1]\times[0,a+\varepsilon)\rangle\cap\langle X, [0,1]\times(a-\varepsilon,1]\rangle$ and for each  $g\in UC([0,1],[0,1])$ with $\Gamma(f)\in\langle [0,1]\times[0,a+\varepsilon)\rangle\cap\langle X, [0,1]\times(a-\varepsilon,1]\rangle$ we have $|a-\max g|<\varepsilon$.

\end{proof}





\section {Definition of a family of fuzzy integrals on any compactum}

 We consider a generalization of t-normed integral for any compactum $X$ based on  the t-overlap and $\co$-general pseudo-grouping functions in this section. Consider any t-overlap function $O$ and any $\co$-general pseudo-grouping function $G$. Let $X$ be a compactum and $\mu\in MX$.

\begin{df} For a  function $f\in  C(X,[0,1])$ we consider an integral defined by the formula $$\int_X^{G,O} fd\mu=G(m)$$ where $m$ is the function defined in Lemma 1 and  call it GO-integral 
\end{df}

We consider the defined integral as a functional on $C(X,[0,1])$, i.e. $I:C(X,[0,1])\to[0,1]$ is a functional defined by the formula  $I(f)=\int_X^{G,O}fd\mu$ for $f\in  C(X,[0,1])$.

\begin{theorem}\label{prop} The functional $I$ satisfies the following properties

\begin{enumerate}
\item $I(1_X)=1$;
\item $I(0_X)=0$;
\item $I(\varphi)\le I(\psi)$ for each functions $\varphi$, $\psi\in C(X,[0,1])$ such that $\varphi\le\psi$;

\end{enumerate}
\end{theorem}

\begin{proof} 1. We have $(1_X)_1=X$. Hence $m(1)=O(\mu(X),1)=1$. Then we obtain $G(m)=1$ by Property 2 from Definition \ref{CGPG}. 

2. We have $(0_X)_0=X$ and $(0_X)_t=\emptyset$ for each $t>0$. Hence $m(0)=O(\mu(X),0)=0$ and $m(t)=O(\mu(\emptyset),t)=O(0,t)=0$ for each $t>0$ by Property 2 from Definition \ref{Ov}. Then we obtain $G(m)=0$ by Property 1 from Definition \ref{CGPG}.

3. Consider any functions $\varphi$, $\psi\in C(X)$ such that $\varphi\le\psi$. The inequality $I(\varphi)\le I(\psi)$ follows from the obvious inclusion $\varphi_t\subset\psi_t$ and the monotonicity of the capacity $\mu$ and the functions $O$ and $G$.
\end{proof}

It is noteworthy that Theorem \ref{prop} implies, for finite $X$, that the GO-integral adheres to the definition of an aggregation function (see for example \cite{Overlap} for more information about aggregation functions).

One of the important problems of the fuzzy integrals theory is characterization of integrals as functionals on some function space (see for example subchapter 4.8 in \cite{Grab} devoted to characterizations of the Choquet integral and the Sugeno integral). Characterizations of t-normed integrals were discussed in \cite{CLM}, \cite{Rad}, \cite{R5} and \cite{R6}. Particulary,  the following characterization of t-normed integral was given in \cite{R5}: a   functional $\mu$ on  $C(X,[0,1])$ satisfies the following properties 

\begin{enumerate}
\item $\mu(1_X)=1$;
\item $\mu(\psi\vee\varphi)=\mu(\psi)\vee\mu(\varphi)$ for each comonotone functions $\varphi$, $\psi\in C(X,[0,1])$;
\item $\mu(c_X\ast\varphi)=c\ast\mu(\varphi)$ for each $c\in[0,1]$ and $\varphi\in C(X,[0,1])$.

\end{enumerate} 
 
if and only if there exists a unique capacity $\nu\in MX$ such that $\mu$ is the t-normed integral with respect to $\nu$.

\begin{problem} Can we extend the characterization provided above to formulate an appropriate description for the GO-integral? Specifically, does the GO-integral adhere to some generalized versions of the second and third properties?
\end{problem}


\begin{thebibliography}{}

\bibitem{BK}  M. Boczek, M. Kaluszka {\em On the extended Choquet-Sugeno-like operator}, International Journal of Approximate Reasoning {\bf 154} (2023), 48--55.

\bibitem{Overlap}  H. Bustincea, J. Fernandez, R. Mesiar, J. Montero, R. Orduna {\em Overlap functions}, Nonlinear Analysis {\bf 72} (2010), 1488--1499.

\bibitem{CLM}  Luis M. de Campos, Mar\'{\i}a T. Lamata and Seraf\'{\i}n Moral {\em A unified approach to define fuzzy integrals}, Fuzzy Sets and Systems {\bf 39} (1991), 75--90.

\bibitem{Ch}  G. Choquet {\em Theory of Capacity,} An.l'Instiute Fourie {\bf 5} (1953-1954), 13--295.



\bibitem{Den} D. Denneberg, Non-Additive Measure and Integral. Kluwer, Dordrecht, 1994.

\bibitem{EK} J.Eichberger, D.Kelsey, {\em Non-additive beliefs and strategic equilibria}, Games Econ Behav {\bf 30} (2000) 183--215.





\bibitem{Gil} I.Gilboa, {\em Expected utility with purely
subjective non-additive probabilities}, J. of Mathematical
Economics {\bf 16} (1987) 65--88.






\bibitem{Grab} Michel Grabisch, Set Functions, Games
and Capacities in Decision
Making. Springer,  2016.

\bibitem{KM} A. Kolesarova, R.Mesiar. Discrete Universal Fuzzy Integrals. In {\em Cornejo, M.E., Harmati, I.A., Koczy, L.T., Medina-Moreno, J. (eds) Computational Intelligence and Mathematics for Tackling Complex Problems 4. Studies in Computational Intelligence, vol 1040.}. p.  Springer. 2023.


\bibitem{PRP} E.P.Klement, R.Mesiar and E.Pap. {\em Triangular Norms}. Dordrecht: Kluwer. 2000.

\bibitem{LMOS} J. Li, R. Mesiar, Y. Ougang, A. Seliga, {\em A new class of decomposition integrals on finite spaces}, International Journal of Approximate Reasoning {\bf 149} (2022) 192--205.

\bibitem{Lin} Lin Zhou, {\em Integral representation of continuous comonotonically additive functionals}, Transactions of the American Mathematical Society {\bf350} (1998) 1811--1822.





\bibitem{NZ} O.R.Nykyforchyn, M.M.Zarichnyi, {\em Capacity functor in the category of compacta}, Mat.Sb. {\bf 199} (2008) 3--26.






\bibitem{Rad} T.Radul, {\em Games in possibility capacities with  payoff expressed by fuzzy integral},  Fuzzy Sets and systems {\bf 434} (2022) 185-197.

\bibitem{R5} T.Radul, {\em Some remarks on characterization of t-normed integrals on compacta},  Fuzzy Sets and sytems {\bf 467} (2023) 108490.

\bibitem{R6} T.Radul, {\em On t-normed integrals with respect to possibility capacities on compacta},  Fuzzy Sets and sytems {\bf 473} (2023) 108716.


\bibitem{Sch} D.Schmeidler, {\em Subjective probability
and expected utility without additivity}, Econometrica {\bf 57} (1989) 571--587.






\bibitem{Sua}  F. Suarez, {\em Familias de integrales difusas y medidas de entropia relacionadas}, Thesis, Universidad
de Oviedo, Oviedo (1983).

\bibitem{Su}  M.Sugeno, {\em Fuzzy measures and fuzzy integrals}, A survey. In Fuzzy Automata and Decision
Processes. North-Holland, Amsterdam: M. M. Gupta, G. N. Saridis et B. R. Gaines editeurs. 89--102. 1977



\bibitem{We1}  S. Weber {\em Decomposable measures and integrals for archimedean t-conorms}, J. Math. Anal. Appl. {\bf 101} (1984), 114--138.

\bibitem{We2}  S. Weber {\em Two integrals and some modified versions - Critical remarks}, Fuzzy Sets and Systems {\bf 20} (1986), 97--105.

\bibitem{OverGPG} Jonata Wieczynski, Rui Paiva, Anderson Cruz, Gra\c{c}aliz P. Dimuro,  Benjam\'{\i}n Bedregal, Carlos L\'{o}pez Molina, {\em Extended families of discrete Sugeno and Choquet integrals and their application
in decision making}, preprint.


\end{thebibliography}
\end{document}